\documentclass[a4paper]{amsart}

\usepackage{amsmath}%
\usepackage{amsfonts}%
\usepackage{amssymb}%
\usepackage{graphicx}
\usepackage{tikz}
\usepackage{enumerate}
\usepackage{rotating}
\usepackage{soul}
\usetikzlibrary{shapes}

\newtheorem{theorem}{Theorem}[section]
\newtheorem{corollary}[theorem]{Corollary}
\newtheorem{lemma}[theorem]{Lemma}
\newtheorem{proposition}[theorem]{Proposition}
\theoremstyle{definition}
\newtheorem{de}[theorem]{Definition}
\newtheorem{example}[theorem]{Example}

\numberwithin{equation}{section}

\def\<#1>{\langle#1\rangle}

\newcommand{\B}{\ensuremath{\mathcal B}}

\begin{document}
\title[The decomposition theorems in Baer $*$-rings]{The decomposition theorems in Baer $*$-rings}
\author{Zbigniew Burdak, Marek Kosiek, Patryk Pagacz, Marek S\l{}oci\'nski}

\address{Zbigniew Burdak, Department of Applied Mathematics, University of
Agriculture, ul. Balicka 253c, 30-198 Krak\' ow,
Poland.}

\email{rmburdak@cyf-kr.edu.pl}

\address{Marek Kosiek, Wydzia\l{} Matematyki i Informatyki,
Uniwersytet Jagiello\'nski, ul. Prof. St. \L{}ojasiewicza 6, 30-348 Krak\'ow, Poland
}\email{Marek.Kosiek@im.uj.edu.pl}

\address{Patryk Pagacz, Wydzia\l{} Matematyki i Informatyki,
Uniwersytet Jagiello\'nski, ul. Prof. St. \L{}ojasiewicza 6, 30-348 Krak\'ow, Poland
 }\email{Patryk.Pagacz@im.uj.edu.pl}

\address{Marek S\l{}oci\'nski, Wydzia\l{} Matematyki i Informatyki,
Uniwersytet Jagiello\'nski, ul. Prof. St. \L{}ojasiewicza 6, 30-348 Krak\'ow, Poland
 }\email{Marek.Slocinski@im.uj.edu.pl}



\thanks{Research was supported by the Ministry of Science and Higher Education of the
Republic of Poland}

\begin{abstract}
We show a general decomposition theorem in Baer $*$-rings. As a consequence the vast majority of decompositions known in the algebra of bounded Hilbert space operators are generalized to Baer $*$-rings. There are also results which are new in the algebra of bounded Hilbert space operators. The model of summands in Wold-S\l oci\' nski decomposition in Baer $*$-rings is given.
\end{abstract}
\maketitle
\section{Preliminaries}


In the recent papers \cite{Bagheri-Bardi2018, BB2} the authors noticed the important role of an algebraic structure in several results on decompositions in the algebra of bounded linear operators on Hilbert spaces. As a consequence they manage to generalize those results to Baer $*$-rings. We show a general decomposition theorem which yields the vast majority of decompositions on hereditary properties. In particular for the algebra of bounded Hilbert space operators they imply known decompositions as well as some new. Since our proof is purely algebraic, the results are formulated in Baer $*$-rings.

Let $R$ be a $*$-ring with unity $1$ and $\tilde{R}\subset R$ denotes the set of all projections (self-adjoint idempotents). Further $S^r:=\{x\in R:sx=0 \text{ for all } s\in S\}$ and similarly defined $S^l$ are the right, the left anihilator of $S\subset R$. Recall that $R$ is a Rickart $*$-ring if a right anihilator of each element is a right principal ideal generated by a projection. The ring $R$ is called a Baer $*$-ring if this property extends on subsets. Then, since $R$ is a $*$-ring, also a left anihilator of each element is a left principal ideal generated by a projection. For a Rickart $*$-ring the set $\tilde{R}$ is a lattice, for a Baer $*$-ring the lattice $\tilde{R}$ is complete.
Since the projection generating anihilator of $x$ in Rickart $*$-ring is unique, we may denote $\{x\}^l=R(1-[x])$ where $[x]\in \tilde{R}$ is called the left projection of $x$. The projection $[x]$ is the minimal one satisfying $[x]x=x$ and $\{x\}^l=\{[x]\}^l$. It follows that $\{x\}^r=(1-[x^*])R$ and $[x^*]$ is the right projection of $x$ (the minimal one satisfying $x[x^*]=x$ and $\{x\}^r=\{[x^*]\}^r$).

For any $p\in\tilde{R}$ the set $pRp$ is a ring with unity $p$ and it is called a corner of $R$. If $R$ is a Rickart $*$-ring or a Baer $*$-ring then their corners are of the same type.
\begin{de}Let $x\in R$ and $p\in\tilde{R}.$
\begin{itemize}
\item A compression of $x$ to $p$ is $pxp\in pRp.$
\item An element $x$ is $p$ invariant if $(1-p)\in\{xp\}^l$. Then $xp(=pxp)$ is a compression of $x$ to $p.$
\item A projections $p$  decomposes $x$ between two summands $$x=xp+x(1-p)=pxp+(1-p)x(1-p)$$ if and only if  $x$ is $p$ and $1-p$ invariant.
\end{itemize}
All the above is defined for subset $S\subset R$ by means that the respective condition holds for any $x\in S$.\end{de}
Note that $p$ decomposes $S$ if and only if $px=xp$ for any $x\in S.$ In other words $p\in S'$ (the commutant of $S$).

More generally a set of pairwise orthogonal projections $\{p_i\}_1^n\subset\tilde{R}$ such that $\sum_{i=1}^np_i=1$ is a unity decomposition (factorization). In Baer $*$-rings a unity decomposition may be infinite where $\sum_{i=1}^\infty p_i:=\sup\{p_1+\dots+p_i:i\in\mathbb{Z}_+\}$. A sequence of pairwise orthogonal projections $\{p_i\}_{i=1}^n$  (possibly $n=\infty$ in Baer $*$-rings) decomposes $x$ if $\{p_i\}_{i=1}^n\subset\{x\}'$ and $x=\sum_{i=1}^n xp_i$ ($:=x\sup\{p_1+\dots+p_i: i\in\mathbb{Z}_+\}$ for $n=\infty$.) If a unity decomposition $\{p_i\}_1^n\subset Z(R)$ (centre of $R$) then any $x\in R$ may be decomposed as $x=x\sum_{i=1}^np_i=\sum_{i=1}^nxp_i$. Hence $R=\sum_{i=1}^np_iRp_i$ is a decomposition of $R$.
It is clear that only projections in the centre provide decompositions of $R$.

\section{Decomposition}


Since now on we assume $R$ to be a Baer $*$-ring. Recall from \cite[Theorem 20]{Kaplansky68} or \cite[Proposition 4.5]{Berberian72}:
\begin{theorem}\label{starcommutant}
A commutant of a $*$-subset of a Baer $*$-ring is a Baer $*$-subring with unambiguous sups and infs.
\end{theorem}
In the introduction we defined $\sum_{i=0}^\infty xp_i:=x\sum_{i=0}^\infty p_i$ for $\{p_i\}$ pairwise orthogonal projections. However, if $xp_i$ are projections then they are pairwise orthogonal and the left hand sum makes sense on its own. Hence, we need to check that the definition causes no ambiguity. It follows by Corollary \ref{infsum}\textit{(1)} provided we check that if $xp_i$ are projections then $x$ is a projection commuting with all $p_i$-s. Indeed, $(xp_i)^2=xp_i=p_ix^*$ yields $xp_i=p_ixp_i$ and $x^2p_i=xp_i$. Hence, $\{p_i\}_{i\ge 0}\subset\{x^2-x\}^r=\{[(x^2-x)^*]\}^r$ and so $1=\sup\{p_i\}\le 1-[(x^2-x)^*]$ implies $x^2=x$. Similarly one can check that $x^*=x$.
\begin{corollary}\label{infsum}Infinite sums admit the following properties:
\begin{enumerate}
\item If a projection $p$ commute with pairwise orthogonal projections $\{p_i\}_{i\ge 1}$ then $p$ commute with $\sum_{i=1}^\infty p_i$ and $p\sum_{i=1}^\infty p_i=\sum_{i=1}^\infty pp_i.$
\item If $p_iq_j=0$ for any $i,j\in\mathbb{Z}_+$ then $(\sum_{i=1}^\infty p_i)( \sum_{j=1}^\infty q_j)=0$ where $\{p_i\}_{i\ge 1},$ $\{q_j\}_{i\ge 1}$ are sets of pairwise orthogonal projections.
\item If $\{p_{(i,j)}\}_{(i,j)\in\mathbb{Z}^2_+}$ is a set of pairwise orthogonal projections then $$\sum_{(i,j)\in\mathbb{Z}^2_+}p_{(i,j)}=\sum_{i=1}^\infty\sum_{j=1}^\infty p_{(i,j)}.$$

\end{enumerate}
 \end{corollary}
\begin{proof}
Recall that $\sum_{i=1}^\infty p_i=\sup\{p_1+\dots+p_i:i\in\mathbb{Z}_+\}$. It is a natural definition. However, a projection majorises a finite sum of pairwise orthogonal projections if and only if it majorises each of its summands. Hence $$\sup\{p_1+\dots+p_i:i\in\mathbb{Z}_+\}=\sup\{p_i:i\in\mathbb{Z}_+\}.$$
Let us point out that \begin{equation}\tag{*}\label{gwiazdka}\text{if }pp_i=0\text{ for }i\ge 1\text{ then }p\sup\{p_i:i\ge 1\}=0.\end{equation}
Indeed, $pp_i=0$ yields $1-p\ge p_i$ for $i\ge 1$. Hence, $1-p\ge\sup\{p_i:i\ge 1\}$ and so $p\sup\{p_i:i\ge 1\}=0$.

Let us show $(1)$. The commutativity we get by Theorem \ref{starcommutant}. It is clear that $pp_i\le p_i\le \sum_{i=1}^\infty p_{i}$ and $pp_i\le p$ yields
$\sum_{i=1}^\infty pp_i\le p\sum_{i=1}^\infty p_i.$ Note that $$p\sum_{i=1}^\infty p_i-\sum_{i=1}^\infty pp_i=pp_{i_0}+p\sum_{\substack{i=1,\\ i\ne i_0}}^\infty p_i-pp_{i_0}-\sum_{\substack{i=1,\\ i\ne i_0}}^\infty pp_i=p\sum_{\substack{i=1,\\ i\ne i_0}}^\infty p_i-\sum_{\substack{i=1,\\ i\ne i_0}}^\infty pp_i.$$ Hence, by (\ref{gwiazdka}) $p\sum_{i=1}^\infty p_i-\sum_{i=1}^\infty pp_i$ is orthogonal to $p_{i_0}$. Since $i_0$ was arbitrary $p\sum_{i=1}^\infty p_i-\sum_{i=0}^\infty pp_i$ is orthogonal to any $p_i$ and by (\ref{gwiazdka}) to $\sum_{i=1}^\infty p_i$. Hence $$0=p\sum_{i=1}^\infty p_i\left(p\sum_{i=1}^\infty p_i-\sum_{i=1}^\infty pp_i\right)=p\sum_{i=1}^\infty p_i-\sum_{i=1}^\infty pp_i,$$ which yields the equality.

For $(2)$
denote $p= \sum_{i=1}^\infty p_i,$ and $q=\sum_{j=1}^\infty q_j$.
 By (\ref{gwiazdka}) $q_jp=0$ for any $j\ge 0$ and hence, again by (\ref{gwiazdka}) $pq=0$.

For $(3)$ denote $p=\sum_{(i,j)\in\mathbb{Z}^2_+}p_{(i,j)}$ and $p_i=\sum_{j=1}^\infty p_{(i,j)}$.  By (2) $\{p_i\}_{i=1}^\infty$ are pairwise orthogonal and $\sum_{i=1}^\infty\sum_{j=1}^\infty p_{(i,j)}=\sum_{i=1}^\infty p_i$ is well defined. For any fixed $i$ the inequality $p\ge p_{(i,j)}$ for any $j$ yields $p\ge p_i$. Hence $p\ge \sum_{i=1}^\infty p_i$. On the other hand $p_{(i,j)}\le p_i\le \sum_{i=1}^\infty p_i$ yields the reverse inequality.

\end{proof}
If $x,y$ commute then not necessarily $[x]$ commute with $y$. Indeed a non unitary isometry $y$  commutes with itself but $[y]y=y\ne y[y]$.
\begin{lemma}\label{cpm}
If $x\in\{y,y^*\}'$ then $[x]\in\{y,y^*,[y]\}'$.
 \end{lemma}
\begin{proof}
Since $y$ commutes with $x$ we get
 $$0=xy-yx=[x]xy-y[x]x=[x]yx-y[x]x=([x]y-y[x])x.$$ In other words $[x]y-y[x]\in\{x\}^l=R(1-[x])$ and so $$[x]y-y[x]=([x]y-y[x])(1-[x])=[x]y-[x]y[x].$$ Reducing $[x]y$ we get $y[x]=[x]y[x]$. Since $x$ commutes also with $y^*$, by similar arguments  $y^*[x]=[x]y^*[x]$. Hence $[x]y=(y^*[x])^*=([x]y^*[x])^*=[x]y[x]=y[x]$. Since $[x]$ is selfadjoint it commutes also with $y^*$.

 We have showed that if $x\in\{y,y^*\}'$ then $[x]\in\{y,y^*\}'$. Replacing $x$ by $y$ and $y$ by $[x]$ we get that if $y\in\{[x]\}'$ then $[y]\in\{[x]\}'$. Since we showed that $y\in\{[x]\}'$ in the first part  the proof is complete.\end{proof}
For the sake of completness let us show some expectable properties.
 \begin{corollary}\label{remcomm}
Let $p$ be a selfadjoint element commuting with an arbitrary $x$. Then $p$ commutes with $[x]$ and $x$ commutes with $[p]$ and $[x]$ commutes with $[p]$. Moreover, if $p$ is a projection then $[px]=p[x]$.
\end{corollary}

 \begin{proof}
 The first part follows by Lemma \ref{cpm}.

  For the second part, note that $p[x]$ is a projection and $p[x]px=px$ since $p\in\{[x]\}'$. Moreover,  if $yp[x]=0$ then $0=yp[x]x=ypx$. Reversely, $0=ypx=yp[x]x$ yields $yp[x]\in\{x\}^l=R(1-[x])$ by which $yp[x]=yp[x](1-[x])=0$. Hence $\{px\}^l=\{p[x]\}^l=R(1-p[x])$ and so, by uniqueness of the left projection $p[x]=[px]$.
\end{proof}

Let us show the main result of the paper. Recall that an element $x$ completely does not have a property $\mathcal{P}$ if and only if for any projection $0\ne p\in\{x\}'$ the compression $px$ does not have the property $\mathcal{P}$.

\begin{theorem}\label{dec}
Let $\{F_i\}_{i\in I}$ be a family of functions $F_i:R^n\mapsto R$ ($n$ not necessarily finite) such that $F_i(q\textbf{x})=qF_i(\textbf{x})$ for any  $\textbf{x}=\{x_j\}_j\in R^n$ and any projection $q\in\textbf{x}'$ where $q\textbf{x}=\{qx_j\}_j$. We say that $\textbf{x}$ has the property $\mathcal{P}_I$ if $\textbf{x}\in\bigcap_{i\in I}\ker F_i$.

 There is a unique projection $p\in\textbf{x}'$ such that $p\textbf{x}$ has the property $\mathcal{P}_I$ and  $(1-p)\textbf{x}$ completely does not have the property $\mathcal{P}_I$.
\end{theorem}
\begin{proof}
Let $p=\sup\{q\in\tilde{R}\cap\textbf{x}', q\le 1-[F_i(\textbf{x})]\text{ for }i\in I\}$. Since projections are selfadjoint, $\tilde{R}\cap\textbf{x}'=\tilde{R}\cap(\textbf{x} \cup \textbf{x}^*)'$. Hence, by Theorem \ref{starcommutant}  $p\in(\textbf{x} \cup \textbf{x}^*)'.$ On the other hand $p\le 1-[F_i(\textbf{x})]$ yields $p\in\{[F_i(\textbf{x})]\}^l=\{F_i(\textbf{x})\}^l$ and so $F_i(p\textbf{x})=pF_i(\textbf{x})=0$ for all $i\in I$. In conclusion $p\textbf{x}$ has the property $\mathcal{P}_I$.

Let $r\in\textbf{x}'$ be a projection such that the compression $r\textbf{x}$ has the property $\mathcal{P}_I.$ Let us show that $r\le p$.
The element $p+r$ is not necessarily an idempotent, but it is selfadjoint, commutes with elements of $\textbf{x}$, and by the assumption, $$(p+r)F_i(\textbf{x})=pF_i(\textbf{x})+rF_i(\textbf{x})=F_i(p\textbf{x})+F_i(r\textbf{x})=0$$
for any $i\in I$. Hence $F_i(\textbf{x})\in\{p+r\}^r=\{[p+r]\}^r$ yields $[p+r]\in\{F_i(\textbf{x})\}^l=\{[F_i(\textbf{x})]\}^l$ for any $i\in I$. Consequently $[p+r]\le 1-[F_i(\textbf{x})]$. On the other hand, by Corollary \ref{remcomm}, $[p+r]\in\textbf{x}'$. Eventually, $[p+r]$ is in the set which  supremum is equal $p$, so $[p+r]\le p$. Hence $1-p\in\{[p+r]\}^r=\{p+r\}^r$ and consequently $0=(p+r)(1-p)=r(1-p)$ yields $r\le p$.

  If $r\in\tilde{R}\cap\mathbf{x}'$ is in the corner $(1-p)R(1-p)$ then $r=(1-p)r=r(1-p)$. Hence $r\le 1-p$. On the other hand, if $r\textbf{x}$ has the property $\mathcal{P}_I$ then, by the previous part of the proof, $r\le p$. In conclusion $r=0$ and so $(1-p)\textbf{x}$ completely does not have the property $\mathcal{P}_I$.

  For uniqueness of $p$ assume that $r$ is a projection that decomposes $\textbf{x}$ between objects having the property $\mathcal{P}_I$ and completely not having it. Since $r\textbf{x}$ has the property $\mathcal{P}_I$, by the previous part of the proof $r\le p$. Hence $p$ and $r$ commute, so $p(1-r)=p-r\le 1-r$ is a projection in $\textbf{x}'$ and $$F_i((p-r)\textbf{x})=(p-r)F_i(\textbf{x})=pF_i(\textbf{x})-rF_i(\textbf{x})=F_i(p\textbf{x})-F_i(r\textbf{x})=0$$ for any $i\in I$. In other words the compression of $(1-r)\textbf{x}$ given by $(p-r)$  have the property $\mathcal{P}_I$. However, by assumption on $r$ the only compression of $(1-r)\textbf{x}$ having the property $\mathcal{P}_I$ is the trivial one. In conclusion $p-r=0,$ so $p$ is unique.
\end{proof}
One can use only a subset $J\subset I$ in Theorem \ref{dec} and get the respective property $\mathcal{P}_J$ and the projection $p_J$. In the following Proposition \ref{manyprop} we show that projections $p_J$ corresponding to various sets $J$ commute to each other. By this commutativity we are able to get decompositions among more than two summands and so gain more detailed descriptions. In the next section we show several applications of this fact. Precisely we extend to Baer $*$-rings several  classical results in the algebra of bounded linear operators on a Hilbert space $\B(H)$. Moreover, we get some new results also in $\B(H)$.
\begin{proposition}\label{manyprop}
Suppose we have functions $F_i:R^n\mapsto R$ for $i\in I_1\cup I_2,$ $\textbf{x}\in R^n$ and projections $p_1, p_2$ decomposing $\textbf{x}$ with a correspondence to properties $\mathcal{P}_{I_1}, \mathcal{P}_{I_2}$, respectively, as in Theorem \ref{dec}.

Projections $p_1, p_2$ commute and their product is $p_{12}$ - the projection corresponding to the property $\mathcal{P}_{I_1\cup I_2}$, as in Theorem \ref{dec}.
\end{proposition}
\begin{proof}
It is clear that $p_{12}\le p_1$ and so $p_{1}(1-p_{12})$ is a well defined projection. Consider an arbitrary projection $q\in\textbf{x}'$ such that $q\le p_{1}(1-p_{12})$. Then $q\le 1-[F_i(\textbf{x})]$ for $i\in I_1$ and there is $j_0\in I_2$ such that $q[F_{j_0}(\textbf{x})]\ne 0$. Hence $F_{j_0}(q\textbf{x})\ne 0$. Indeed,  $0=F_{j_0}(q\textbf{x})=qF_{j_0}(\textbf{x})$ yields $q[F_{j_0}(\textbf{x})]=0$ which is not true.
Since $q$ was arbitrary, $p_{1}(1-p_{12})\textbf{x}$ completely does not have the property $\mathcal{P}_{I_2}$. Consequently $p_{1}(1-p_{12})\le 1-p_2$ and so $p_1p_2=p_2p_1=p_{12}.$
\end{proof}
\section{Applications}

In this section we derive several decompositions from Theorem \ref{dec}.
The condition $pF_i(\textbf{x})=F_i(p\textbf{x})$ makes $\mathcal{P}_I$ a hereditary property (i.e. if $\textbf{x}$ has the property $\mathcal{P}_I$ then any compression of $\textbf{x}$ given by a commuting projection has $\mathcal{P}_I$ as well). By hereditarity, $(1-p)\textbf{x}$ completely does not have the property $\mathcal{P}_I$ if and only if $p$ is the maximal (so unique) projection such that $p\textbf{x}$ has the property. Hence, some statements in the section claim the existence of the maximal projection which is equivalent to the existence of the corresponding decomposition. Let us give a little leeway that for non hereditary property the maximality does not imply the uniqueness of the corresponding decomposition -- there may exist different decompositions between a part having the property and the one completely not having it. The reason is that the maximality of the projection may be considered only as a maximal element of some chain without uniqueness. For example, the property of being a bilateral shift is non hereditary. There may exist different bilateral shift parts of the same unitary operator on a Hilbert space. We skip details since it requires Spectral Theorem and is far from the subject of the article.

In the section we recall or adopt from the the algebra of bounded linear operators on Hilbert spaces several properties of a Baer $*$-ring elements. Let us start with the basic ones. Recall that an element $x\in R$ is called normal, a partial isometry, an isometry, a unitary element if $xx^*=x^*x, xx^*x=x, x^*x=1, x^*x=xx^*=1$ respectively.
\begin{theorem}\label{C1}
For any $x$ in a Baer $*$-ring there are the maximal projections $p_n, p_p, p_i, p_u$ commuting with $x$ such that:
\begin{itemize}
\item $p_nx$ is normal,
\item $p_px$ is a partial isometry,
\item $p_ix$ is an isometry,
\item $p_ux$ is unitary.
\end{itemize}
\end{theorem}
\begin{proof}
The result follows from Theorem \ref{dec} where:
\begin{itemize}
\item $F(x)=xx^*-x^*x$ for $p_n$,
\item $F(x)=x-xx^*x$ for $p_p$,
\item $F(x,1)=1-x^*x$ for $p_i$,
\item $F_1(x,1)=1-x^*x, F_2(x,1)=1-xx^*$ for $p_u$.
\end{itemize}
Let us comment on an extra argument $1$ which appeared in the last formulas. The compression $p\textbf{x}$ is considered in the corner $pRp$ where the unity is $p$. Consequently, whenever a unity plays any role in the formula, it is added as an extra argument, to be replaced by $p$ in the corresponding compression. Hence, the condition $pF(\textbf{x})=F(p\textbf{x})$ is satisfied.
\end{proof}
The next result is formulated for a general element of a Baer $*$-ring, but it can be viewed in the context of Halmos-Wallen-Foia\c s result on power partial isometries \cite[Theorem]{Halmos70}.
\begin{corollary}
For any $x$ in a Baer $*$-ring there is a unity decomposition $$p_u+p_{pi}+p_{pci}+p_r=1$$ such that $p_u,p_{pi},p_{pci},p_r\in\{x\}'$ and
\begin{itemize}
\item $p_ux$ is unitary,
\item $p_{pi}x$ is a pure isometry,
\item $p_{pci}x$ is a pure co-isometry,
\item $p_rx,p_rx^*$ are completely not isometric.
\end{itemize}
\end{corollary}
\begin{proof}
Indeed, let $p_u, p_i$ be as in Theorem \ref{C1} for $x$ while $p_{ci}$ be an isometric projection calculated for $x^*$ in Theorem \ref{C1}. By Proposition \ref{manyprop} $p_ip_{ci}=p_{ci}p_i=p_u$. Hence $p_{pi}=p_i(1-p_u)$ and $p_{pci}=p_{ci}(1-p_u)$ are well defined and orthogonal to each other. It remains to define $p_r=1-p_u-p_{pi}-p_{pci}$ which, by orthogonality to $p_i$ and $p_{ci}$, compress $x$ and $x^*$ to completely non isometric elements.
\end{proof}
For isometries the result is finer, the last part is described as truncated shifts \cite{BB2, Halmos70}.

Theorem \ref{dec} may be successfully applied to pairs (more generally sets) of elements. It works well, nevertheless a property describes a relation between/among elements (f.e. commutativity) or characterizes elements (f.e. normality). The following result on double commutativity may be modified to a commutativity. Recall that a pair of elements $(x, y)$ doubly commute if $x\in\{y, y^*\}'$.
 \begin{theorem}\label{dcdec}
For any pair $(x,y)$ of arbitrary elements in $R$ there is a unique projection $p\in\{x,y\}'$ such that $(px, py)$ doubly commute and  $((1-p)x, (1-p)y)$ completely not doubly commute.
\end{theorem}
\begin{proof}
It follows from Theorem \ref{dec} for  $F_1(x,y)=xy-yx, F_2(x,y)=xy^*-y^*x$.
\end{proof}

Much wider class are compatible pairs. The concept of compatibility was introduced for isometries on Hilbert spaces by K. Hor\'ak; V. M\"uller in \cite{Muller89}. It naturally extends to general pairs of elements in Baer $*$-rings.
\begin{de}
A pair $(x,y)$ is compatible if $\{[x^m]: m\ge 1\}\subset\{[y^n]: n\ge 1\}'$.
\end{de}
The following corollary is obvious for isometries in $\B(H)$, while in Baer $*$-rings it follows by Lemma \ref{cpm}.
 \begin{corollary}\label{dcarcomp}
 Any doubly commuting pair is compatible.
 \end{corollary}
 An example of compatible, completely non doubly commuting pair is $(x,x)$ where $x$ is a non unitary isometry. Another examples can be found in papers on operators on Hilbert spaces \cite{BKPS2017, BKS2015, Muller89}. In particular in \cite{BKS2015} there is given a precise model of a commuting, compatible pair.

The next result shows a decomposition of an arbitrary pair between a compatible pair and a completely non compatible pair.
\begin{theorem}\label{compdec}
For any $x,y\in R$ there is $p\in\{x,y\}'$ such that $(px, py)$ is a compatible pair and  $((1-p)x, (1-p)y)$ is completely non compatible.
\end{theorem}
\begin{proof}
It follows from Theorem \ref{dec} for $F_{m,n}(x,y)=[x^m][y^n]-[y^n][x^m]$ for $m,n\in\mathbb{Z}_+$ where $F_{m,n}(px,py)=pF_{m,n}(x,y)$ by Corollary \ref{remcomm}.
\end{proof}
As a conclusion of Corollary \ref{dcarcomp} and Theorems \ref{dcdec} and \ref{compdec} we get a decomposition of an arbitrary pair of commuting elements among three compressions.
\begin{theorem}\label{dec3}
For any pair $(x,y)$ of arbitrary elements in $R$ there are unique projections $p, q\in\{x,y\}'$ where $p\le q$ such that
\begin{itemize}
\item $(px, py)$ doubly commute,
\item $(q(1-p)x, q(1-p)y)$ are compatible, completely non doubly commuting,
\item $((1-q)x,(1-q)y)$ is completely non compatible.
\end{itemize}
\end{theorem}

A very rich class of examples of Baer $*$-rings are bounded operators on Hilbert spaces $\B(H)$. Theorem \ref{dcdec} in $\B(H)$ is known. However, to the authors knowledge, compatibility was defined only for isometries so far. Hence Theorem \ref{compdec} is new also in $\B(H)$.

The compatibility does not imply commutativity. We give an example. Recall that two projections are equivalent if there is a partial isometry having them as the left and the right projection.
 \begin{example}
Let $p_{i,j}$ be a set of pairwise orthogonal and equivalent projections for $i,j=1,2,3$ and $x_{i,j}, y_{i,j}$ be partial isometries for $i,j=1,2$ such that $[x_{i,j}^*]=[y_{i,j}^*]=p_{i,j},\, [x_{i,j}]=p_{i+1,j},\, [y_{i,j}]=p_{i,j+1}.$ Define $$x=\sum_{i,j=1,2}x_{i,j},\quad y=\sum_{i,j=1,2}y_{i,j}\quad .$$

Let us check that $x, y$ are compatible. Since $[x^*]=\sum_{i,j=1,2}[x_{i,j}^*]$ and similarly for $[x], [y^*], [y]$ we get $$[x^*]=[y^*]=\sum_{i,j=1,2}p_{i,j},\;[x]=\sum_{\substack{i=2,3\\j=1,2}}p_{i,j},\; [y]=\sum_{\substack{i=1,2\\j=2,3}}p_{i,j}.$$
By orthogonality of projections $p_{i,j}$ one may check that $x^2=\sum_{j=1,2}x_{2,j}x_{1,j}$ and $y^2=\sum_{i=1,2}y_{i,2}y_{i,1}$ and $x^n=y^n=0$ for $n\ge 3.$
Hence $[x^2]=p_{3,1}+p_{3,2},\; [y^2]=p_{1,3}+p_{2,3},\; [x^n]=[y^n]=0$ for $n\ge 3$ which yields compatibility.

Consider $y'=y_{1,1}+y_{1,2}+uy_{2,1}+y_{2,2}$ where $u\ne p_{2,2}$ is a partial isometry such that $[u]=[u^*]=p_{2,2}$ (in other words compression of $u$ to $p_{2,2}Rp_{2,2}$ is unitary, not unity). Note that $[y^n]=[y'^n]$ for any $n\ge 0$ so $x, y'$ are compatible as well as $x,y$ are. Let us show that at most one of pairs $x,y$ and $x,y'$ may commute. Indeed $yxp_{1,1}=y_{2,1}x_{1,1}p_{1,1}\ne uy_{2,1}x_{1,1}p_{1,1}=y'xp_{1,1}$ while $xyp_{1,1}=x_{1,2}y_{1,1}p_{1,1}=xy'p_{1,1}.$ Hence at most one of equalities $xy=yx, xy'=y'x$ may hold.
\end{example}

One may ask about a quaternary decomposition with respect to commutativity and compatibility as in Theorem \ref{dec4} below. The answer is affirmative, but not obvious even in $\B(H)$. It follows from Proposition \ref{manyprop}.

\begin{theorem}\label{dec4}
For any pair $(x,y)$ of arbitrary elements in $R$ there is a unique identity decomposition among $p_{11}, p_{1,0}, p_{01}, p_{00}\in\{x,y\}'$ such that
\begin{itemize}
\item $(p_{11}x, p_{11}y)$ commute and are compatible,
\item $(p_{10}x, p_{10}y)$ commute and are completely non compatible,
\item $(p_{10}x, p_{10}y)$ completely do not commute and are compatible,
\item $(p_{10}x, p_{10}y)$ completely do not commute and are completely non compatible.
\end{itemize}
\end{theorem}
\begin{proof}
Let $I_1=\{1\}$ and $F_1(x,y)=xy-yx$ and $I_2=\mathbb{Z}_+^2$ and $F_{m,n}(x,y)=[x^m][y^n]-[y^n][x^m]$ as in the proof of Theorem \ref{compdec}. Obviously the corresponding properties $\mathcal{P}_{I_1}$ and $\mathcal{P}_{I_2}$ defined as in Theorem \ref{dec} are commutativity and compatibility, respectively. On the other hand, by Proposition \ref{manyprop} projections $p_{cm}, p_{cp}$ corresponding to $\mathcal{P}_{I_1}, \mathcal{P}_{I_2}$ commute. Hence $p_{11}=p_{cm}p_{cp}, p_{10}=p_{cm}(1-p_{cp}), p_{01}=(1-p_{cm})p_{cp}, p_{00}=(1-p_{cm})(1-p_{cp})$ provides the decomposition required in the statement.\end{proof}

 Let $\mathcal{P}$ be a property characterizing indyvidual elements (f.e. normality). Recall that a set $S$ completely does not have the property $\mathcal{P}$ (f.e. is completely not normal) if for any $0\ne p\in S'$ there is at least one $x\in S$ such that $pxp$ does not have the property (at least one $pxp$ is not normal).
We extend results of Theorem \ref{C1} on subsets. We show the decomposition with respect to normality. Other results may be proved similarly.
\begin{corollary}\label{normalfamily}
Let $S\subset R$ where $R$ is a Baer $*$ -ring. There is a maximal projection $p\in S'$ such that $pS$ is normal (a set of normal elements).
\end{corollary}
\begin{proof}
It is enough to take $F_s:R^{S}\ni \textbf{x} \mapsto x_s^*x_s-x_sx_s^*$ for any $s\in S$ and apply Theorem \ref{dec}.
\end{proof}

Let us finish this section by a generalization of Wold, Helson-Lowdenslager, Suciu result \cite[Theorem 3]{Suciu68}. For those reason we extend the concept of a quasi-unitary semigroup of isometries to Baer $*$-rings.
\begin{de}
Let $G$ be an abelian group  and $S\subset G$ be a semigroup such that $S\cap S^{-1}=\{1\}$ and $SS^{-1}=G$.
Denote $\{x_s\}_S$ a semigroup of isometries in a Baer $*$-ring $R$ (i.e. $x_1=I, x_sx_r=x_{sr}$).

We call a semigroup $\{x_s\}_S$ quasi-unitary if $\sup\{[x_g^*x_s]: g^{-1}s\notin S^{-1}\}=1.$ A semigroup is purely quasi-unitary if it is quasi-unitary and completely non unitary.

A completely non quasi-unitary group is called strange.
\end{de}
\begin{theorem}\label{isogroup}
Let $\{x_s\}_S$ be a semigroup of isometries in a Baer $*$-ring $R$. There are projections $p_u,p_{qu},p_s\in\{x_s\}'_S$ such that $p_u+p_{qu}+p_s=1$ and
\begin{itemize}
\item $p_ux_s$  is unitary for every $s\in S$,
\item $\{p_{pqu}x_s\}_S$ is purely quasi-unitary,
\item $\{p_sx_s\}_S$ is strange.
\end{itemize}
\end{theorem}
\begin{proof}
Define $F(\textbf{x})=1-\sup \{[x_g^*x_s]: g^{-1}s\notin S^{-1}\}$ where $\textbf{x}\in R^S, \textbf{x}=\{x_s\}_S.$ Note that by Corollary \ref{remcomm}, $pF(\textbf{x})=F(p\textbf{x})$ for any projection $p\in \textbf{x}'$. Hence, by Theorem \ref{dec} we get a projection $p_{qu}$ which is the maximal one compressing the semigroup to a quasi-unitary semigroup. Similarly like in Corollary \ref{normalfamily} we consider a family of functions $F_s:R^{S}\ni \textbf{x} \mapsto 1-x_sx^*_s$ and get a projection $p_u$. It is clear that $p_u\le p_{qu}$ and so $p_{pqu}=p_{qu}(1-p_u)$ is a well defined projection compressing the semigroup to a purely quasi-unitary semigroup. Clearly $p_s=1-p_{qu}$ compress the semigroup to a strange semigroup.
\end{proof}

\section{Multiple canonical decomposition}

Consider a property of an individual element. Assume there is a pair $(x, y)$ such that each of its elements admits a decomposition between summand having the property and the one completely not having it. We may usually find also a decomposition of the pair $(x,y)$ between the pair having the property and the one completely not having it as in Corollary \ref{normalfamily} for example. However, the fact that the pair completely does not have the property does not say much about individual elements in the pair. Indeed, consider as an example the property of being normal. A normal element and a completely not normal element as well as two completely not normal elements form completely not normal pairs. Hence a pair completely not having some property requires a finer description. Wold, Helson-Lowdenslager, Suciu result recalled in the previous section is one of the first attempts of characterizations of this type. The best would be a quaternary decomposition, as defined:
\begin{de}
A canonical decomposition of a pair $(x,y)$ with respect to a property $\mathcal{P}$ characterizing single elements $x, y$ is a quaternary decomposition
$$p_{11}+p_{10}+p_{01}+p_{00}=1$$ where $p_{11}, p_{10}, p_{01}, p_{00}\in\{x,y\}'$ are such that
\begin{itemize}
\item each of $p_{11}x, p_{11}y$ has the property $\mathcal{P}$,
\item $p_{10}x$ has the property $\mathcal{P}$, $p_{10}y$ completely does not have the property $\mathcal{P}$,
\item $p_{01}x$ completely does not have the property $\mathcal{P}$, $p_{01}y$ has the property $\mathcal{P}$,
\item each of $p_{00}x, p_{00}y$ completely does not have the property $\mathcal{P}$.
\end{itemize}
\end{de}
Unfortunately, a general pair may not admit a canonical decomposition. Let us explain why Proposition \ref{manyprop} does not work for canonical decompositions.  Consider once again the property of being normal. By Proposition \ref{manyprop} projections $p_x, p_y$ corresponding (in the sense of Theorem \ref{dec}) to $F_x(x,y)=x^*x-xx^*, F_y(x,y)=y^*y-yy^*$ do commute. Hence $p_xp_y$ is a projection. It can be checked that it is a maximal projection where both compressions are normal. However, $p_x(1-p_y)$ compress $x$ to a normal element but $p_x(1-p_y)y$ is not necessarily a completely not normal element. Indeed, there may exist a projection $0\ne q\le p_x(1-p_y)$ commuting with $y$ where $qy$ in normal but $q$ does not commute with $x$. To be precise, in the decomposition of $x$ we consider projection corresponding to $F_x(x)=x^*x-xx^*$ instead of $F_x(x,y)=x^*x-xx^*$. The formula is the same. The difference is that the respective supremum is taken among projection commuting only with $x$ in the first case and with both $(x,y)$ in the second case. Hence the projection corresponding to $F_x(x)$ may majorize the one corresponding to $F_x(x,y)$. Let us formulate the result similar to \cite[Corollary (2.3)]{CatPtakSzym}.
\begin{proposition}\label{canequiv}
Let $F_i:R\mapsto R$ for $i\in I$ be a family of functions, where $R$ is a Baer $*$-ring. An element $r\in R$ is said to have the property $\mathcal{P}$ if $r\in\bigcap_{i\in I}\ker F_i$.  If $x,y\in R$ are such that $pF_i(x)=F_i(px),\; qF_i(y)=F_i(qy)$ for $i\in I$ and any projection $p\in\{x\}',\; q\in\{y\}'$, respectively then:
\begin{itemize}
\item there are maximal projections $p_x\in\{x\}'$ and $p_y\in\{y\}'$ such that $p_xx, p_yy$ have the property $\mathcal{P},$
\item there are maximal projections $q_x, q_y\in\{x,y\}'$ such that $q_xx, q_yy$ have the property $\mathcal{P},$
\item $q_x\le p_x,\; q_y\le p_y$.
\end{itemize}

Moreover, the following conditions are equivalent:
\begin{itemize}
\item $(x,y)$ admits a canonical decomposition with respect to the property $\mathcal{P}$,
\item $p_x, p_y\in\{x,y\}',$
\item $p_x=q_x, p_y=q_y.$
\end{itemize}
\end{proposition}
\begin{proof} In fact the first part has been explained before the proposition. Precisely, the existence of $p_x, p_y$ follows from Theorem \ref{dec} for $\{F_i\}_{i\in I}.$ Define $$F_{i1}(x,y):=F_i(x),\quad F_{i2}(x,y):=F_i(y)$$ for $i\in I$. Then Proposition \ref{manyprop} for $I_1=I\times \{1\}, I_2=I\times \{2\}$
yields the existence of $q_x, q_y$ which commute. To see that $q_x\le p_x$ recall from the proof of Theorem \ref{dec} that $$p_x=\sup\{q\in\tilde{R}\cap \{x\}':q<1-[F_i(x)]\}.$$
Note that $q_x\le 1-[F_{i1}(x,y)]=1-[F_i(x)]$ and obviously $q_x\in\tilde{R}\cap \{x\}'$. Hence, $q_x$ belongs to the set above, so $q_x\le p_x$.  Similarly $q_y\le p_y$.

For the second part, denote by $p_{11}+p_{10}+p_{01}+p_{00}=1$ the canonical decomposition of the pair $(x,y)$. If it exists, then $p_x=p_{11}+p_{10}, p_y=p_{11}+p_{01}$ so they commute with both $x,y$. If $p_x\in\{x,y\}'$ then, similarly as we showed $q_x\le p_x$ we may show the reverse inequality. If $p_x=q_x, p_y=q_y$ then $p_x, p_y$ commute. One can check that $p_{11}=p_xp_y,\, p_{10}=p_x(1-p_y),\, p_{01}=(1-p_x)p_y,\, p_{00}=(1-p_x)(1-p_y)$ is the canonical decomposition with respect to the property $\mathcal{P}.$
\end{proof}
Recall  that any pair of doubly commuting operators in $\B(H)$ admits a canonical decomposition with respect to any hereditary property \cite[Corollary 2.4]{CatPtakSzym}. Unfortunately, the proof is based on von Neumann algebras,  precisely Double Commutant Theorem. In the case of Baer $*$-rings such a result, if it is correct, requires a different proof. By Theorem \ref{starcommutant} if $p_x$  may be obtained as a supremum of projections commuting with $y$ then it commutes with $y$ as well (the notation as in Proposition \ref{canequiv}). Such a condition is used to show the existence of a canonical decomposition of a pair of doubly commuting isometries with respect to unitarity in \cite{Bagheri-Bardi2018}. However, it is not the only way. One can imagine a set of projections $p_i\in \{x\}'$ such that $p_iy=yp_{i+1}$. Then $p_i\notin \{y\}'$ but $\sup\{p_i\}\in\{y\}'$.

Recall that the decomposition of an isometry with respect to unitarity is called Wold decomposition and the corresponding canonical decomposition of a pair Wold-S\l oci\' nski decomposition. Such results in Baer $*$-rings are showed in \cite{Bagheri-Bardi2018}. Recall that an isometry $x$ is a unilateral shift if $[x^n(1-[x])]$ are pairwise orthogonal for $n\ge 0$ and $\sum_{n=0}^\infty[x^n(1-[x])]=1$.
\begin{theorem}[{\cite[Theorem 2.4]{Bagheri-Bardi2018}}] Let $x$ be an isometry in a Baer $*$-ring $R$. Then there is a unique projection $p_u\in \{x\}'$ such that,
\begin{itemize}
\item the compression $p_ux$ is unitary and,
\item the compression $(1-p_u)x$ is a unilateral shift.
\end{itemize}
\end{theorem}
The existence of Wold decomposition  follows from Theorem \ref{C1}. The important advantage of the result in \cite{Bagheri-Bardi2018} is that a completely non unitary isometry is described as a unilateral shift. The generalization of Wold-S\l oci\'nski decomposition to Baer $*$-rings \cite[Theorem 3.2]{Bagheri-Bardi2018} is only a decomposition. The models of all summands in $\B(H)$ are known and can be generalized to Baer $*$-rings.
\begin{theorem}\label{WSdec}
Let $(x, y)$ be a pair of doubly commuting isometries in a Baer $*$-ring. There are unique $p_{uu},p_{us},p_{su},p_{ss}\in\{x,y\}'$ such that $$p_{uu}+p_{us}+p_{su}+p_{ss}=1$$
 and
\begin{itemize}
\item  $(p_{uu}x, p_{uu}y)$ is a pair of unitary elements,
\item  $p_{us}x$ is unitary  and $p_{us}y$ is a unilateral shift and $$p_{us}x=\sum_{i=0}^\infty p_{us}x[y^i(1-[y])],$$
\item  $p_{su}x$ is a unilateral shift and $p_{su}y$ is unitary and $$p_{su}y=\sum_{i=0}^\infty p_{su}y[x^i(1-[x])],$$
\item $(p_{ss}x, p_{ss}y)$ is a pair of unilateral shifts and $$p_{ss}=\sum_{m,n\ge 0}[p_{ss}x^my^n(1-[x])(1-[y])].$$
    \end{itemize}
    \end{theorem}
    \begin{proof}
    The existence of the decomposition is shown in \cite[Theorem 3.2]{Bagheri-Bardi2018}. It was not emphasized that it is unique. However, since unitarity is a hereditary property it is unique. A precise proof is a consequence of uniqueness in Theorem \ref{dec}. Indeed, by Corollary \ref{normalfamily} there is $p_{uu}$ the maximal (so unique) projection compressing $(x,y)$ to unitary elements (normal isometries are unitary elements). On the other hand $p_{uu}+p_{us}$ is the maximal projections compressing $x$ to a unitary element which is also unique by Theorem \ref{C1}. Hence $p_{us}$ is unique as well as, by similar arguments, $p_{su}$. Consequently also $p_{ss}$ is unique.

    Since $x,y$ doubly commute, $x$ commute with $[y^n]$ for any $n$ as well as $y$ commute with $[x^m]$ for any $m\ge 0$ by Lemma \ref{cpm}.

    Let us describe compression $(p_{us}x, p_{us}y)$ (and similarly $(p_{su}x, p_{su}y)$).  By the above and since $x$ commute also with $p_{us}$ it doubly commutes with $p_{us}y^m(1-[y])$ and, by Lemma \ref{cpm}, with $[p_{us}y^m(1-[y])]$. Note that $(p_{us}y)^i(1-[p_{us}y])=p_{us}y^i(1-[y])$. Hence, and since $p_{us}y$ is a unilateral shift, $\sum_{i=0}^\infty [p_{us}y^i(1-[y])]=p_{us}$ (recall that unity in $p_{us}Rp_{us}$ is $p_{us}$). Since by Corollary \ref{remcomm} $[p_{us}y^i(1-[y])]=p_{us}[y^i(1-[y])]$ we get  $$p_{us}x=\sum_{i=0}^\infty p_{us}x[p_{us}y^i(1-[y])]=\sum_{i=0}^\infty p_{us}x[y^i(1-[y])]$$ which is a decomposition of $p_{us}x$.

    Let us show the last part. Since $p_{ss}x, p_{ss}y$ are unilateral shifts and by Corollary \ref{infsum}  \begin{align*}p_{ss}=p_{ss}^2&=\left(\sum_{m=0}^\infty [p_{ss}x^m(1-[x])]\right)\left(\sum_{n=0}^\infty [p_{ss}y^n(1-[y])]\right)\\&=\sum_{m,n\ge 0}[p_{ss}x^m(1-[x])][p_{ss}y^n(1-[y])].\end{align*} It has left to show that $$[p_{ss}x^m(1-[x])][p_{ss}y^n(1-[y])]=[p_{ss}x^my^n(1-[x])(1-[y])]$$ which by Corollary \ref{remcomm} is equivalent to
     $$p_{ss}[x^m(1-[x])][y^n(1-[y])]=p_{ss}[x^my^n(1-[x])(1-[y])].$$
  Note that $x^m(1-[x]), y^n(1-[y])$ are partial isometries. Moreover, as $(1-[x])(1-[y])$ is a projection, also $x^my^n(1-[x])(1-[y])$ is a partial isometry. Recall that for a partial isometry $z\in R$ there is $[z]=zz^*$. Hence one can check that
   \begin{align*}
   [x^m(1-[x])]&=[x^m]-[x^{m+1}],\\
   [y^n(1-[y])]&=[y^n]-[y^{n+1}],\\
   [x^my^n(1-[x])(1-[y])]&=([x^m]-[x^{m+1}])([y^n]-[y^{n+1}])\end{align*}
     which finishes the proof.
    \end{proof}

\bibliographystyle{abbrv}
\bibliography{baer}
\end{document}